\documentclass[11pt,a4paper, oneside]{amsart}
\usepackage[utf8]{inputenc}
\usepackage[english]{babel}
\usepackage{amsfonts,amsmath,amssymb,amsthm,enumitem, mathrsfs}
\usepackage[left=2.4 cm,right=2.4 cm,top=2.4 cm,bottom=2.3 cm]{geometry}
\usepackage{mathtools}

\usepackage{color, xcolor}
\definecolor{LinkBlue}{RGB}{0,92,170}
\definecolor{CiteGreen}{RGB}{0,128,90}
\definecolor{UrlPurple}{RGB}{120,0,160}
\PassOptionsToPackage{
  unicode,
  colorlinks=true,
  linkcolor=LinkBlue,
  citecolor=CiteGreen,
  urlcolor=UrlPurple,
  pdfborder={0 0 0}
}{hyperref}
\usepackage{ifpdf}
\ifpdf
    \usepackage[pdftex]{graphicx}
    \DeclareGraphicsExtensions{.png,.pdf,.jpg}
\else
   \usepackage[dvips]{graphicx}
   \DeclareGraphicsExtensions{.eps}
\fi
\graphicspath{{.}{figures/}}
\numberwithin{equation}{section}

\usepackage{doi}
\hypersetup{
  colorlinks=true,
  linkcolor=LinkBlue,
  citecolor=CiteGreen,
  urlcolor=UrlPurple,
  pdfborder={0 0 0}
}
\usepackage{comment}
\usepackage[p,osf]{cochineal}
\usepackage[scale=.95,type1]{cabin}
\usepackage[zerostyle=c,scaled=.94]{newtxtt}

\newtheorem{theorem}{Theorem}[section]
\newtheorem{lemma}[theorem]{Lemma}
\newtheorem{proposition}[theorem]{Proposition}

\newtheorem{definition}[theorem]{Definition}
\newtheorem{remark}[theorem]{Remark}
\newtheorem*{theorem*}{Theorem}
\newtheorem*{lemma*}{Lemma}
\newtheorem*{proposition*}{Proposition}
\newtheorem*{corollary*}{Corollary}%\pagefooter{}{\thepage}{{\small\sc [\today]}}
\newtheorem{hypothesis}[theorem]{Hypothesis}

\def\R{\mathbb{R}}

\newcommand\dtv{\mathrm{d}_{\mathrm{TV}}}
\newcommand{\I}{\mathrm{I}}
\newcommand{\II}{\mathrm{II}}
\newcommand{\III}{\mathrm{III}}

\def\N{\mathbb{N}}
\def\EE{\mathbb{E}}

\def\D{\mathscr{D}}
\def\bD{\beta^{\mathrm{low}}_{p,d}}

\newcommand{\dkl}{\mathrm{d}_{\mathrm{KL}}}

\def\LM#1{\hbox{\vrule width.2pt \vbox to#1pt{\vfill \hrule width#1pt
height.2pt}}}
\def\LL{{\mathchoice {\>\LM7\>}{\>\LM7\>}{\,\LM5\,}{\,\LM{3.35}\,}}}
\def\restr{{\LL}}
\renewcommand{\phi}{\varphi}

\def\1{\mathbf{1}}

\def\XXint#1#2#3{{\setbox0=\hbox{$#1{#2#3}{\int}$ }
\vcenter{\hbox{$#2#3$ }}\kern-.57\wd0}}

\def\eps{\varepsilon}
\def\lt{\left}
\def\rt{\right}
\def\les{\lesssim}
\def\ges{\gtrsim}

\DeclareMathOperator*{\argmin}{\arg\!\min}

\newcommand{\sqa}[1]{\left[ #1 \right]}

\newcommand{\Wb}{\mathscr{W}b}

\def\W{\mathscr{W}}
\begin{document}

\title{The Wasserstein cost of importance sampling}
\author{Simon Coste}
\address{Laboratoire de Probabilités, Statistique et Modélisation, Universit\'e Paris Cit\'e, Paris, France}
\email{simon.coste@u-paris.fr}
\author{Michael Goldman}
\address{Centre de Mathématiques Appliquées de l’École Polytechnique, Palaiseau, France}
\email{michael.goldman@cnrs.fr}
\date{\today}

\begin{abstract}Importance sampling (IS) consists in biasing samples from a distribution $f$ towards another distribution $g$. 
    Concretely, given samples $X_i$ from $f$, the IS measure is $$\hat{g}_n = \frac{1}{Z_n}\sum_{i=1}^n \frac{g(X_i)}{f(X_i)} \delta_{X_i},$$ with $Z_n = \sum_{i=1}^n \frac{g(X_i)}{f(X_i)}$. The random measure $\hat{g}_n$ approximates $g$, and is used in many contexts ranging from Monte Carlo integration to Bayesian inference. We show that, in high dimension ($d \geqslant 3$), the Wasserstein cost $\W_p^p(\hat{g}_n, g)$ has order $n^{-p/d}$ in expectation, i.e.
    $$\bD\int gf^{-p/d}\leqslant \liminf_{n \to \infty} n^{p/d} \mathbb{E}[\W_p^p(\hat{g}_n, g)] \leqslant \limsup_{n \to \infty} n^{p/d} \mathbb{E}[\W_p^p(\hat{g}_n, g)] \leqslant\beta_{p,d} \int g f^{-p/d}$$
    where $\bD\leqslant \beta_{p,d}$ are constants depending only on $p$ and $d$, which are equal for $p=2$ and conjectured to be equal for any $p\geqslant 1$. Our results are valid for all $p\geqslant 1$ and $d\geqslant 3$.
    
    In the case where $\bD = \beta_{p,d}$, we show that the asymptotically optimal sampling distribution $f^*$ for importance sampling is not equal to $g$ but to a tempered version of $g$, namely $f^* \propto g^{d/(p+d)}$, which is reminiscent of Zador’s theorem in the domain of measure quantization. 
\end{abstract}

\maketitle 

%\setcounter{tocdepth}{1}
%\tableofcontents
% only one level for sections 

\section{Introduction} 

Let $f,g$ be two probability densities over $\mathbb{R}^d$, supported on the same set. For any measurable function $\varphi$, one can write 
\begin{equation}\label{eq:is}
    \mathbb{E}_{Y \sim g}[\varphi(Y)] = \mathbb{E}_{X \sim f}\left[\varphi(X) \frac{g(X)}{f(X)}\right].
\end{equation}
In particular, for any $\varphi$, if we have samples $X_i$ from $f$, we can approximate the expectation $\mathbb{E}_{Y \sim g}[\varphi(Y)]$ by the empirical average 
\begin{equation}\label{eq:is_0}
    \frac{1}{n}\sum_{i=1}^n \varphi(X_i) \frac{g(X_i)}{f(X_i)}.
\end{equation}
This random variable has mean $\mathbb{E}_{Y \sim g}[\varphi(Y)]$, and if the samples are independent, by the law of large numbers, it converges to this expectation as $n\to \infty$, which gives an estimator of the integral $\mathbb{E}_{Y \sim g}[\varphi(Y)]$ using only samples from $f$. However, \eqref{eq:is_0} implicitly assumes that we can evaluate the density \(g\) itself. In many applications, the target is only available \emph{up to a multiplicative constant}, and the normalization constant $\int g$ could be impossible to compute. This is the usual situation for unnormalized measures, e.g. in Bayesian inference where the posterior is proportional to likelihood times prior. In that case, the weights \(g(X_i)/f(X_i)\) in \eqref{eq:is_0} cannot be computed.

Self-normalized importance sampling resolves this issue by replacing \eqref{eq:is_0} with a ratio in which any unknown multiplicative constant cancels, 
\begin{equation}\label{eq:is_sn}
    \frac{\sum_{i=1}^n \varphi(X_i) g(X_i)/f(X_i)}{\sum_{i=1}^n g(X_i)/f(X_i)}. 
\end{equation}
While this estimator is in general biased, the law of large numbers for \eqref{eq:is_0} still ensures that the numerator converges to $\int \varphi g$ and the denominator converges to $\int g$, hence the estimator converges to $\mathbb{E}_{X \sim g}[\varphi(X)]$ even if $g$ is only known up to a normalization constant. 

The estimator \eqref{eq:is_sn} can be seen as the integral of $\varphi$ against the self-normalized importance sampling measure $\hat{g}_n$, which motivates the following definition. From now on, we will restrict to the cube $(0,1)^d$.

\begin{definition}[self-normalized importance sampling measure]Let $f,g$ be two probability densities on $(0,1)^d$. Let $(X_i)$ be a sequence of iid random variables with distribution $f$. The \emph{self-normalized importance sampling measure} for $g$ based on the samples from $f$ is defined as 
    \begin{equation}
        \hat{g}_n = \frac{1}{Z_n}\sum_{i=1}^n W_i \delta_{X_i}, 
    \end{equation}
    where $W_i = g(X_i)/f(X_i)$ are the importance weights and $Z_n$ is the normalization constant
    \begin{equation}
        Z_n = \sum_{i=1}^n W_i.
    \end{equation}
\end{definition}

Importance sampling and its variants are used in a large variety of contexts, and it is crucial to understand the accuracy of these methods and how it behaves depending on the sample size $n$. Many works were devoted to this topic, as we will review in the next section. It is generally acknowledged that the accuracy of \eqref{eq:is_sn} depends on the distribution and concentration of the importance weights $W_i = g(X_i)/f(X_i)$: if the sampling density $f$ and the target density $g$ are equal, these weights are all equal to 1. Conversely, any significant discrepancy between $f$ and $g$ will lead to a poor concentration of the importance weights (they could have large variance, or be heavy-tailed), resulting in a very poor accuracy of \eqref{eq:is_sn}. 

In 2018, a landmark paper by Chatterjee and Diaconis \cite{chatterjee_diaconis} studied how $\int \varphi d\hat{g}_n$ is close to $\int \varphi dg$ for a fixed function $\varphi$ with unit square-norm. They essentially showed that there is a cutoff at the sample size $n = e^{\dkl(g\mid f)}$, where $\dkl(g|f)$ is the Kullback-Leibler divergence between $f$ and $g$, in the sense that if $n \gg e^{\dkl(g\mid f)}$, then 
$$\mathbb{P}\left(\left|\int \varphi d\hat{g}_n - \int \varphi dg\right|>\varepsilon\right) = O(\varepsilon), $$
while if $n \ll e^{\dkl(f|g)}$, then one can find (at least) one square-integrable function $\varphi$ such that $\int \varphi d\hat{g}_n = 1$ with overwhelming probability, but $\int \varphi dg \leqslant \varepsilon$. This is often interpreted as a curse of dimensionality phenomenon, because in high dimension, the divergence $\dkl(g\mid f)$ can grow linearly with $d$ in many contexts (see \cite[Proposition 15]{bruna2024provable} for a prototypical example).

In this work, we are interested in a more quantitative study of the estimation error using the self-normalized importance sampling measure. Our results are global in the sense that we directly measure the $p$-Wasserstein distance $\W_p(\hat{g}_n, g)$ for $p\geqslant 1$, and prove that it has order $n^{-1/d}$. 

\section{Main result}

\subsection{Asymptotic behavior of the Wasserstein distance}

The $p$-Wasserstein distance is defined for pairs of Borel measures $\mu$ and $\nu$ on $(0,1)^d$ by the formula
\begin{equation}\label{def:Wp}\W_p^p(\mu, \nu)= \inf_{\Pi \in \mathscr{C}(\mu, \nu)}\int |x - y|^p \mathrm{d}\Pi(x, y),\end{equation}
where $\mathscr{C}(\mu, \nu)$ is the set of couplings between $\mu$ and $\nu$. A coupling between $\mu$ and $\nu$ is a measure $\Pi$ on the product space with marginals $\mu$ and $\nu$. Note that we did not impose any normalization on $\mu$ and $\nu$, which may or may not be probability measures---but they must have the same mass.

\begin{hypothesis}\label{hyp:W}
    Our results hold for differentiable densities $f, g$ which are bounded from above and below and with a bounded gradient. For simplicity, we denote by $c>0$ a constant such that all the quantities $$|g|_\infty, |f|_\infty, |1/g|_\infty, |1/f|_\infty,|\nabla g|_\infty, |\nabla f|_\infty, |f/g|_\infty$$ are smaller than $c$. 
\end{hypothesis}

\begin{theorem}Let $f,g$ be two probability densities on $(0,1)^d$ satisfying Hypothesis \ref{hyp:W} and let $\hat{g}_n$ be the self-normalized importance sampling measure for $g$ based on $n$ iid samples from $f$.
    Then, for any $p\geqslant1$, there is a constant $\beta_{p,d} \in ]0,\infty[$ such that 
    \begin{equation}
        \limsup_{n \to \infty} n^{p/d} \mathbb{E}\left[\W_p^p\left(\hat{g}_n, g\right) \right]\leqslant \beta_{p,d}\int_{(0,1)^d} g(x) f(x)^{-p/d}\mathrm{d} x.
    \end{equation}
    Additionally, there is a constant $0<\bD \leqslant \beta_{p,d}$ such that 
    \begin{equation}\label{eq:lower_bound_Wp}
        \liminf_{n \to \infty} n^{p/d} \mathbb{E}\left[\W_p^p\left(\hat{g}_n, g\right) \right]\geqslant \bD\int_{(0,1)^d} g(x) f(x)^{-p/d}\mathrm{d} x.
    \end{equation}
\end{theorem}

\subsection{Remarks}

{~} \bigskip

 \paragraph{A}The constants $\beta_{p,d}$ and $\bD$ are respectively defined in \eqref{def:beta} and \eqref{def:betalow}. It is conjectured that they are equal. This was only proved in the case where $p=2$ and $d\geqslant 3$ in \cite{huesmann2024asymptotics}. 
\bigskip

 \paragraph{B} Our proof has the decisive advantage of being valid for the whole range of $p\geqslant 1$ and $d\geqslant 3$. Other simpler proofs (like the one in \cite{BaBo13}) could work, but only for tighter ranges like $p\leqslant d/2$.

\bigskip

\paragraph{C}  
 The hypothesis that the common support of $f$ and $g$ is the cube $(0,1)^d$ could be relaxed to cover compact sets with Lipschitz boundary at the price of some technicalities, see \cite[Proposition 5.1]{goldman2024optimal}.

\bigskip

\paragraph{D}    When $g=f$, this theorem reduces to \cite[Theorem A.2]{goldman2024optimal}, 
    where it is proved that $\W^p_p(\hat{g}_n, g)$ is of order $n^{-p/d}\beta_{p,d}\int g^{1 - p/d}$. The extra term $\int g^{1 - p/d}$ is equal to 1 if $g\equiv 1$, hence this term accounts for the non-uniformity of 
     $g$. Since $x\mapsto x^{1 - p/d}$ is convex if $p\geqslant d$ and concave if $p < d$, Jensen's inequality yields
\begin{equation}\label{eq:g^1-p/d}
    \int_{(0,1)^d} g(x)^{1 - p/d} \leqslant 1 \quad \text{ iff } p < d.
\end{equation}
The interpretation of \eqref{eq:g^1-p/d} is that, when $p < d$ and $g$ is nonuniform, the IS error is \emph{smaller} than if $g$ were uniform. The regime $p< d$ is the typical one in high dimension; in practice, one often uses $p=1$ or $p=2$.

\subsection{Asymptotically optimal sampling distribution}

It might sound natural that, when approximating $g$ by IS, using reweighted samples from a 
distribution $f\neq g$, the error should be higher, i.e. the extra term $\int g f^{-p/d}$ should be greater
 than $\int g^{1 - p/d}$. Again, this is not true in general. Suppose for example that the sampling distribution is uniform, ie $f=1$. Then, \eqref{eq:g^1-p/d} shows that $\int g^{1 - p/d} > 1=\int g f^{-p/d}$ if $p \geqslant d$. In other words, \emph{the IS error using real samples from $g$ is greater than the IS error using uniform samples from $g$}. So it is only natural to ask, for a fixed $g$, what is the sampling distribution $f$ that minimizes the error $\W_p(\hat{g}_n, g)^p$. Our results do not provide the exact asymptotic of $\W_p(\hat{g}_n, g)^p$ as $n \to \infty$, because the two constants $\beta_{p,d}$ and $\bD$ are not known to be equal --- but they are generally conjectured to be equal. Under this conjecture, we would have 
$$\W_p^p(\hat{g}_n, g) \sim n^{-p/d}\beta_{p,d}\int_{(0,1)^d} g(x) f(x)^{-p/d}\mathrm{d} x$$
and it is natural to minimize this asymptotic equivalent with respect to $f$. The following proposition gives the minimizer. 

\begin{proposition}The solution to the problem 
$$f^* \in \argmin_{f \geqslant 0, \int f = 1} J_g(f), \qquad J_g(f) = \int g f^{-p/d}$$
is found at 
\begin{equation}\label{eq:f_star}
f^*(x) = \frac{1}{\int g^{\frac{d}{p+d}}}g(x)^{\frac{d}{p+d}}.
\end{equation}
The minimum value is then 
\begin{equation}\label{eq:J_f_star}
    J_g(f^*) = \left(\int g^{\frac{d}{p+d}}\right)^{\frac{p+d}{d}}.
\end{equation}
\end{proposition}
\begin{proof}
We set $a = p/d$ and define the Lagrangian of the problem, $\mathscr{L}(f, \lambda) = J_g(f) + \lambda (\int f - 1)$. The function $x^{-a}$ is convex, so $J_g$ is itself convex, hence the first-order conditions for optimality are sufficient. These conditions read $ - a g f^{-a-1} + \lambda = 0 \quad \text{and} \quad \int f = 1$, whose solution is any density $f$ which must be proportional to $g^{1/(1+a)}$, hence \eqref{eq:f_star}. 
\end{proof}

It is noticeable that if $g$ is not uniform, then $f^* \neq g$. Moreover, since $g$ is admissible, we necessarily have $J_g(f^*) \leqslant J_g(g)$. This result might feel counter-intuitive: when using importance sampling to approximate $g$, the best sampling distribution is not $g$ itself if we measure the error by the transport cost with $\W^p_p$. For example, in 3d problems where the error is measured in the euclidean distance $(p=2)$, the best sampling distribution is (proportional to) $g^{3/5}$, which is very different from $g$. 

\begin{remark}[quantization and Zador's theorem]
    Given a measure $\mu$, the \emph{quantization of $\mu$} with $n$ points is the solution of the following optimization problem:
    \begin{equation}\label{prob:quantization}
        \mathscr{Q}_n(\mu) = \min_{(w_i, x_i)_{1\leqslant i\leqslant n}} \W_p^p\left(\mu, \sum_{i=1}^n w_i \delta_{x_i}\right)
    \end{equation}
    where the minimum is taken over $n$-tuples of points $x_i$ and positive weights $w_i$ such that $\sum_{i=1}^n w_i = 1$. It is known under the name of Zador's theorem (see for example \cite{graf2000foundations}, Theorem 6.2 and Theorem 7.5) that 
    $$\lim_{n \to \infty} n^{p/d}\mathscr{Q}_n(g) = c_{d}\left(\int g^{\frac{d}{p+d}}\right)^{\frac{p+d}{d}}$$
    where $c_{d}$ is a constant; this is, up to a constant, the exact value of the problem \eqref{eq:J_f_star}. Moreover, under some mild conditions, the empirical distribution of the $x_i$ in the arg min of \eqref{prob:quantization} can be shown to converge weakly to \eqref{eq:f_star}. It thus comes as no surprise that the optimal sampling measure $f^*$ for IS is the same as the asymptotically optimal quantization measure.
    \end{remark}

\begin{remark}[Effective Sample Size] There is an extensive literature on \emph{Effective Sample Sizes} (ESS), see \cite{elvira2021advances} for a review. Given a method for sampling a target distribution $g$ using $n$ samples, one can compare any metric (like the variance) of this sampling method with what would be the same metric if we had \emph{real} samples from $g$. 

Consider IS with samples from $f$ and target density $g$, and let $\hat{g}_n$ be the importance sampling measure defined in \eqref{eq:is_sn}; let $\bar{g}_m$ be the empirical measure of $m$ genuine samples from $g$. Asymptotically, the equation $\mathbb{E}[\W^p_p(\hat{g}_n, g)] = \mathbb{E}[\W^p_p(\bar{g}_m, g)]$ becomes 
$$m^{-p/d}\int g^{1-p/d} \sim n^{-p/d}\int g f^{-p/d},$$
which is solved for 
\begin{equation}
    m_{\mathrm{ESS}} \sim n \times \left(\frac{\int g^{1 - p/d}}{\int g f^{-p/d}}\right)^{\frac{d}{p}}. 
\end{equation}
Interestingly, the remark below the main theorem suggests that this ESS is not always smaller than $n$; having $m_{\mathrm{ESS}}>n$ would mean that using importance sampling with $n$ proposal samples from $f$ is actually beneficial over using $n$ samples from the real density $g$. 
\end{remark}

\subsection{Proof outline}
The proof of this theorem spans Sections \ref{sec:proof_main_unif} and \ref{sec:lowerbound}, and adapts the general strategy developed in a series of papers studying optimal matchings between point sets and measures, \cite{goldman2021convergence},\cite{caglioti2024subadditivity}, \cite{goldman2024optimal}. 
Appendices \ref{appendix:concentration_inequalities} and \ref{appendix:functional_inequalities} contain classical results on concentration inequalities and functional inequalities used in the proof. 

\subsection*{Acknowledgements} The authors thank Dario Trevisan for useful discussions and suggestions. Part of this research was supported by the ANR project STOIQUES. This work was supported by a public grant from the Fondation Mathématique Jacques Hadamard.

\section{Related work and perspectives}

\subsection{Literature review}

Importance sampling (IS) is a classical Monte Carlo device, tracing back at least to the early work of Kahn \cite{kahn1953methods}, and now ubiquitous in statistics and applied probability. We refer to the surveys by Agapiou \emph{et al.} \cite{agapiou2017importance} or Elvira \emph{et al.} \cite{elvira2021advances} for a broad overview.

%algebra : \cite{diaconis2026counting} for counting the number of group orbits by marrying the Burnside process with importance sampling. 

A large body of work on IS aims at stabilizing weights or diagnosing when stabilization is needed. Truncated importance sampling \cite{ionides2008truncated} and Pareto-smoothed importance sampling \cite{vehtari2024pareto} are widely used practical fixes, and recent work studies regimes with unbounded weight functions, like \cite{deligiannidis2025importancesamplingindependentmetropolishastings}. IS has become a core tool in variational inference, see for example the importance-weighted VAE \cite{burda2016importanceweightedautoencoders} or the recent work \cite{cherief2025asymptotics}. IS is also used in the training or sampling of generative models in machine learning, see \cite{bruna2024provable,grenioux2024stochastic,pavon2021data,carbone2023efficient,sabour2025testtimescalingdiffusionsflow} for recent examples.

These works typically measure accuracy through function-by-function errors (variance, CLT, concentration for a fixed test function) or through scalar summaries of the weights (ESS and its variants), which are indispensable in practice but do not directly quantify how close the \emph{random measure} \(\hat{g}_n\) is to the target distribution as an object. In general, there are few papers on a rigorous study of the accuracy of IS beyond the landmark paper \cite{chatterjee_diaconis}, and more recently \cite{guo2025complexity} but with a focus on the normalization constant estimation in a dynamic setting (like AIS or using Jarzynski's inequality). By treating self-normalized IS as a \emph{weighted empirical measure} and analyzing it in \(p\)-Wasserstein distance, we obtain sharp scaling laws for \(\W_p(\hat{g}_n,g)\). This provides a genuinely different guarantee: it controls
 the full output of IS as a probability measure (for example, if $p=1$, \emph{all} Lipschitz test functions simultaneously), and it makes explicit how the proposal $f$ enters through a transport-relevant functional of $f$ and $g$.

Our contribution is to bridge IS with the geometric theory of empirical measures in optimal transport. Rates for Wasserstein distances between empirical measures and their limits are connected to optimal matching problems, starting from Ajtai--Koml\'os--Tusn\'ady \cite{AKT84} and further developed in, e.g.,  \cite{dobric1995asymptotics, FoGu15, BaBo13, AmGl19, goldman2024optimal,caglioti2024subadditivity, ambrosio2022quadratic,ledoux2019optimal,goldman2024optimal}. Another line of research focused on general metric (Polish) spaces, see \cite{dudley1969speed,boissard2014mean,
weed2019sharp, dedecker2026concentrationempiricalmeasurewasserstein} ; but we are not aware of any work on importance sampling or weighted empirical measures in Wasserstein distance. 

\subsection{Future work}

\subsubsection{Unbounded weights}

The main limitation of our work lies in the restriction that $f$ and $g$ have full support on $(0,1)^d$ and are bounded from above and below. This is a strong assumption: in this case the weights $W = g(X)/f(X)$ are bounded and cannot be heavy-tailed, for example. But their distribution is crucial to assess the performance of an importance sampling strategy. In fact, $\mathbb{E}_{X \sim f}[W] = 1$, but the variance is given by 
\begin{equation}
    \mathbb{E}_{X \sim f}[W^2] - 1 =  \int \frac{g(x)^2}{f(x)}dx -1 = \mathbb{E}_{X \sim g}[W] - 1.
\end{equation}
This quantity can be infinite, leading to a high variability of \eqref{eq:is_0}, which would be dominated by a few large weights. Unfortunately, our method does not apply to this case, and extending it to the case where $f,g$ are allowed to have unbounded or zero values is hard. See however \cite{huesmann2024asymptotics,caglioti2024subadditivity} for some results in this direction.

\subsubsection{Annealed importance sampling}

Annealed Importance Sampling \cite{neal2001annealed} is a popular refinement of IS. The end goal is to estimate means of functionals with respect to a density $g=g_0$, starting from samples $X_i$ from a density, say, $g_n$. Given a chain of densities $g_0, g_1, \ldots, g_n$, typically $g_k$ is proportional to $g_0^{\beta_k}g_n^{1-\beta_k}$ for some $1=\beta_1 > \beta_2 > \ldots > \beta_n=0$, the annealed importance sampling estimator is defined as follows: starting at $X_i^{(n)}=X_i$, one samples $X_i^{(n-1)}$. This is typically done by running a Markov chain starting from $X_i^{(n)}$, whose invariant measure is $g_{n-1}$. The chain is run for several steps; the resulting sample $X_i^{(n-1)}$ might \emph{not} be exactly distributed according to $g_{n-1}$. Then one samples $X_i^{(n-2)}$ by running a Markov chain starting from $X_i^{(n-1)}$, whose invariant measure is $g_{n-2}$, and so on, until one reaches $X_i^{(0)}$. The annealed importance sampling measure is then defined as
    $$
    \mu = \frac{1}{n}\sum_{i=1}^n W_i \delta_{X_i^{(0)}}
    $$
    where 
    \begin{equation}
        W_i = \frac{g_{n-1}(X_i^{(n-1)})}{g_n(X_i^{(n-1)})}\times \frac{g_{n-2}(X_i^{(n-2)})}{g_{n-1}(X_i^{(n-2)})}\times \cdots \times \frac{g_0(X_i^{(0)})}{g_1(X_i^{(0)})}.
    \end{equation}
One difference with classical importance sampling is that the importance weights do not only depend on $X_i^{(0)}$, the points effectively sampled at the end of the chain. To study AIS, we would need a version of our main result which is suitable for \emph{random} weights $W$ which only satisfy 
$$\mathbb{E}[W \mid X] = \frac{g(X)}{f(X)}.$$
Our preliminary work on this general problem suggests that the constant $\beta_{p,d}$ in our main result would now depend on the whole distribution of $W$ in a non-trivial way.

\subsubsection{Girsanov change-of-measure}
Up to now, we have mostly been interested in the high but finite-dimensional setting; but the problem of finding estimates for infinite-dimensional quantities is also of interest. Let $(B_n)$ be a sequence of independent Brownian motions from $[0,1]$ to $\mathbb{R}^d$. We equip $C=\mathscr{C}([0,1],\mathbb{R}^d)$ with the topology of uniform convergence. The Wasserstein distance between Gaussian measures
on $C$ and the empirical measure $\mu_n = n^{-1}\sum_{i=1}^n \delta_{B_i}$ was studied in\footnote{Their results apply to the Wiener measure and also to the fractional Brownian measure and the Brownian sheet. }  \cite{boissard2014mean} (see also references therein), although the authors only seem to prove upper bounds. But in applied stochastic calculus, many problems revolve around importance sampling through the Girsanov identity. In its simplest form, this theorem states that if $B$ is a Brownian motion and $(h_t)_{t \in [0,1]}$ is a process adapted to the filtration generated by $B$, then
for any functional $\varphi$, 
$$\mathbb{E}[\varphi(B)Z_T] = \mathbb{E}[\varphi(X)]$$
where $Z_T$ is the terminal value of the ``change of measure" process
 $$Z_t = \exp\!\left(-\int_0^t h_s\,dB_s - \frac{1}{2}\int_0^t h_s^2\,ds\right),$$ 
 and $X = (X_t)_{t \in [0,1]}$ is the diffusion process defined by $dX_t = dB_t + h_t dt$. This allows one to compute observables from $X$ using only samples from $B$, thus falling exactly in the framework of importance sampling. We plan on extending our results to this setting in the near future.

\section{Background and notation}

\subsection*{Notation}
For any Borel set $S$, we denote by $|S|$ its volume and $\mathbf{1}_S$ its indicator function. For any real quantities $A,B$, the notation $A \les B$ or equivalently $A = O(B)$ means that there exists a constant $C>0$ (possibly depending on $p$ and $d$ only) such that $A \leqslant C B$. We write $A \sim B$ if $A \les B$ and $B \les A$. 

For a measure $\mu$ and a set $S$, we denote by $\mu\restr S$ the restriction of $\mu$ to $S$.

\subsection*{Elementary inequalities}
Throughout the proof, we will repeatedly use the two elementary inequalities: 
\begin{equation}\label{ineq:elementary}\forall a, b\geqslant 0, \quad\forall p\geqslant 1, \qquad (a+b)^p \leqslant (1+\varepsilon)a^p + c_p\varepsilon^{1-p}b^p\end{equation} for some constant $c_p$; and for all $a_1, \ldots, a_k \geqslant 0$, \begin{equation}\label{ineq:elementary_sum}\forall p\geqslant 1, \qquad (a_1 + \dots + a_k)^p \leqslant k^p (a_1^p + \dots + a_k^p).\end{equation}

\subsection*{The Wasserstein distance and its properties }

For two measures $\mu$ and $\nu$ on the same measurable space $S$ and with the same mass, the Wasserstein distance of order $p$ is defined by
\begin{equation}
    \W_p(\mu, \nu)= \left(\inf_{\Pi \in \mathscr{C}(\mu, \nu)}\int_{S \times S} |x - y|^p \mathrm{d}\Pi(x, y)\right)^{1/p},\end{equation}
where $\mathscr{C}(\mu, \nu)$ is the set of couplings between $\mu$ and $\nu$. A coupling between $\mu$ and $\nu$ is a measure $\Pi$ on the product space $S \times S$ with marginals $\mu$ and $\nu$. Note that we did not impose any normalization on $\mu$ and $\nu$, which may or may not be probability measures---but they must have the same mass. 
The quantity $\W_p^p(\mu, \nu)$ is called the $p$-Wasserstein \emph{cost} between $\mu$ and $\nu$, and it is linear in the mass of $\mu$ and $\nu$ in the sense that 
\begin{equation}\label{wasserstein_linearity}
    \W_p^p(\alpha \mu, \alpha \nu) = \alpha \W_p^p(\mu, \nu).
\end{equation}

If $S$ is bounded, then clearly the term $|x-y|^p$ is bounded by $\mathrm{diam}(S)^p$, so that $\int_{S \times S} |x - y|^p \mathrm{d}\Pi(x, y) \leqslant \mathrm{diam}(S)^p \Pi(S \times S)$, and since $\Pi$ is a coupling between $\mu$ and $\nu$, we have $\Pi(S \times S) = \mu(S) = \nu(S)$. Consequently, 
\begin{equation}\label{wasserstein_bound_bounded_domain}
    \W_p^p(\mu, \nu) \leqslant \mathrm{diam}(S)^p \mu(S).
\end{equation}

If $R$ is a subset of $S$ such that $\mu(R) = \nu(R)$, we denote by $\W_{p,R}(\mu, \nu)$ the Wasserstein distance of order $p$ between $\mu\restr R$ and $\nu\restr R$. If $\mathscr{R}$ is a collection of disjoint subsets of $S$ such that $\cup_{R \in \mathscr{R}} R = S$ and such that $\mu(R) = \nu(R)$ for all $R \in \mathscr{R}$, then for any collection of couplings $\Pi_R$ between $\mu\restr R$ and $\nu\restr R$ for $R \in \mathscr{R}$, the measure $\Pi = \sum_{R \in \mathscr{R}} \Pi_R$ is a coupling between $\mu$ and $\nu$. Consequently, 
\begin{equation}
    \W_{p,S}^p(\mu, \nu)\leqslant \sum_{R \in \mathscr{R}} \W_{p,R}^p(\left.\mu\right|_R, \left.\nu\right|_R).
\end{equation}
This property is called \emph{subadditivity} of the Wasserstein distance. 

\section{Proof of the upper bound}\label{sec:proof_main_unif}

\subsection{The setup}

Let $g$ be a probability density on $(0,1)^d$, and let $X$ be a random variable on $(0,1)^d$ with probability density $f$. We set
\begin{equation}\label{def:W}W = \frac{g(X)}{f(X)}.\end{equation}
Then, for any test function $\varphi$, 
\begin{equation}
    \mathbb{E}_{X \sim g}[\varphi(X)] = \mathbb{E}_{X \sim f}[\varphi(X) W].
\end{equation}
The density $f$ of $X$ will often be called the \emph{sampling density} while $g$ is called the \emph{target density}.

Several of our estimates are formulated in terms of the relative second moment of $W$, denoted 
\begin{equation}\label{def:v}v_A = \frac{\mathrm{Var}(W \mathbf{1}_{X \in A})}{\mathbb{E}[W \mathbf{1}_{X \in A}]^2}\qquad \forall A \subset (0, 1)^d. \end{equation}
In this paper, we only use very crude bounds on $v_A$. Indeed, since we assumed in Hypothesis \ref{hyp:W} that $W \leqslant c$, it is easily seen that $v_A \leqslant \frac{c}{\int_A g}$; and since $g \geqslant c^{-1}$, this is smaller than $c^2 |A|^{-1}$: 
\begin{equation}\label{bound_v}\forall A, \qquad v_A \les |A|^{-1}. 
\end{equation}

Let $n$ be an integer and $(X_i)_{i \in [n]}$ be a sequence of independent random variables with  density $f$, with their weights $W_i = g(X_i)/f(X_i)$. We set
\begin{equation}
    \mu_n = \sum_{i=1}^n W_i \delta_{X_i}
\end{equation}
the empirical measure of the $X_i$, weighted by the $W_i$, and
\begin{equation}
    Z_n = \sum_{i=1}^n W_i
\end{equation}
the normalization constant, so that $Z_n^{-1}\mu_n$ is a probability measure. When $g=f$, the weights $W_i$ are all equal to 1. Under the additional hypothesis that $g\equiv 1$, the main result of \cite{goldman2021convergence} is the existence of the following limit, which will be our definition of $\beta_{p,d}$:
\begin{equation}\label{def:beta}\limsup_{n \to \infty} n^{p/d} \mathbb{E}\W_p^p\left(\frac{1}{n}\sum_{i=1}^n \delta_{X_i}, 1\right) = \beta_{p,d}.\end{equation}
Our goal is to prove that, under Hypothesis \ref{hyp:W}, 
    \begin{equation}\label{eq:main_0}
        \limsup_{n \to \infty} n^{p/d} \mathbb{E}\W_p^p\left({Z_n}^{-1}\mu_n, g\right) \leqslant \beta_{p,d}\int_{(0,1)^d} g(x) f(x)^{-p/d} \mathrm{d} x.
    \end{equation}
    We will first prove a non-normalized version of \eqref{eq:main_0}: 
    \begin{equation}\label{eq:main_1}
        \limsup_{n \to \infty} n^{p/d-1} \mathbb{E}\mathscr{W}_p^p\left(\mu_n, Z_n g\right) \leqslant \beta_{p,d}\int_{(0,1)^d} g(x) f(x)^{-p/d} \mathrm{d} x. 
    \end{equation}
We will then deduce \eqref{eq:main_0} from \eqref{eq:main_1}.

\subsection{Dyadic decomposition}

We divide the cube $Q^0=(0,1)^d$ dyadically into $2^{dk}$ dyadic cubes $Q^k_i$ with side length \begin{equation}\label{def:l}\ell_k=2^{-k}\end{equation} and we note $\D_k = \{Q^k_i : i \in [2^{dk}]\}$ the set of these cubes. Note that the diameter of any $Q \in \D_k$ is
\begin{equation}\label{diam}
    \mathrm{diam}(Q) = \sqrt{d}\,2^{-k} \sim \ell_k.
\end{equation}
We will study how the points $X_i$ are distributed in these cubes and how this distribution changes when we change the scale $k$. For this, we note
\begin{equation}
    \label{def:nu}\nu_n = \sum_{i=1}^n \delta_{X_i}
\end{equation}
the \emph{unweighted empirical measure} of the $X_i$. For every measurable set $Q\subset (0,1)^d$,  we define
\begin{align}\label{def:kappa_gamma_delta}
    &\kappa_Q=\frac{\mu_n(Q)}{\int_Q g}, &&\gamma_Q = \frac{\mu_n(Q)}{\nu_n(Q)} &&\delta_Q = \frac{\mu_n(Q)}{\int_Q f}. 
\end{align}

In the case where $\nu_n(Q)=0$, then we also have $\mu_n(Q)=0$ and we set $\gamma_Q=0$. 
For any scale $k$, we define the measure
\[
 \lambda^k=\sum_{Q \in \D_k} \kappa_Q \mathbf{1}_Q g. 
\]
Note that, since $\kappa_{Q^0} = Z_n$, we clearly have $\lambda^0 = Z_n g$. We will go down to a maximal scale $K$ defined by 
\begin{equation}\label{def:K}
    K = \lfloor d^{-1} \log_2 n \rfloor - \omega \quad \qquad \omega = \lfloor\log_2 ((\log n)^\alpha)\rfloor
\end{equation}
for some $\alpha > 0$ --- it will be clear later on that any $\alpha \geqslant 3$ works. For simplicity, we could choose $\alpha=3$. The side length $\ell_K$ of any cube $Q \in \D_K$ is thus equal to 
\begin{equation}\label{eq:ell_K}
    \ell_K = 2^{-K} \sim \frac{2^{\omega}}{n^{1/d}} = \frac{(\log(n))^{\alpha}}{n^{1/d}}
\end{equation}
and the volume of any cube $Q \in \D_K$ is equal to 
\begin{equation}\label{eq:vol_DK}
    |Q| = \ell_K^d \sim \frac{(\log(n))^{\alpha d}}{n}.
\end{equation}

\subsection{Reduction to the dyadic decomposition at the maximal scale}

\begin{lemma}\label{lem:reduction_to_maximal_scale}For any  $p\geqslant 1$, 
    \begin{equation}\label{ineq:mainterm}
        \limsup_{n \to \infty} n^{p/d - 1}\EE\lt[\W_{p}^p(\mu_n, Z_n g)\rt] \leqslant (1+\eps) \limsup_{n \to \infty} \sum_{Q \in \D_K} n^{\frac{p}{d}-1}\EE\lt[\W_{p,Q}^p(\mu_n , \kappa_{Q} g)\rt].
    \end{equation}
\end{lemma}

This lemma, whose proof is postponed to Section \ref{proofs_technical_lemmas}, reduces the problem to the study of local expectations like $\EE\lt[\W_{p,Q}^p(\mu_n , \kappa_{Q} g)\rt]$ for any fixed dyadic cube $Q$, which is what we do now. We will simply write $\gamma, \delta, \kappa$ instead of $\gamma_Q, \delta_Q, \kappa_Q$. On the dyadic cube $Q$, we will compare $\mu_n$ with the measures $\gamma \nu_n, \delta f$ and $\kappa g$ which all have the same mass $\mu_n(Q)$ on $Q$. By the triangle inequality, 
\begin{align*}
 \W_{p,Q}^p(\mu_n , \kappa g)\le \lt( \W_{p,Q}(\mu_n , \gamma \nu_n)+ \W_{p,Q}(\gamma \nu_n, \delta f)+ \W_{p,Q}( \delta f,\kappa g)\rt)^p.
\end{align*} 
It will appear thereafter that the leading contribution is the `flattened' term $\W_{p,Q}^p(\gamma \nu_n, \delta f)$,
 which transports a point measure with all weights equal to $\gamma$ towards the reference measure $\delta f$. 
Using \eqref{ineq:elementary} and averaging, we have
\begin{align}
 \nonumber \mathbb{E}[\W_{p,Q}^p(\mu_n , \kappa g)] &\leqslant (1+\eps) \mathbb{E}[\W_{p,Q}^p(\gamma \nu_n, \delta f)]+ \frac{c_p}{\eps^{p-1}}\lt( \mathbb{E}[\W_{p,Q}^p(\mu_n , \gamma \nu_n)]+ \mathbb{E}[\W_{p,Q}^p( \delta f,\kappa g)]\rt)\\
 &=: (1+\varepsilon)(\I) + \frac{c_p}{\varepsilon^{p-1}}[(\II) + (\III)].\label{threeterms}
\end{align}
We bound each of these three terms separately, and uniformly in $Q \in \D_K$. The proofs are again postponed to Section \ref{proofs_technical_lemmas}.

\begin{lemma}\label{lem:I}
    If $n$ is large enough, then 
    $$\forall Q \in \D_K, \qquad \mathbb{E}[\W_{p,Q}^p( \gamma \nu_n,\delta f)] \leqslant (1 + \varepsilon)^{3+p} \beta_{p,d} n^{1 - p/d} \int_{Q} g f^{-p/d}.$$
\end{lemma}
Due to Hypothesis \ref{hyp:W}, the term $\int_Q gf^{-p/d}$ has order $|Q|=(\log n)^{\alpha d}/n$, so the right-hand side is of the order of 
  $\ges n^{ - p/d}\log(n)^{\alpha d}$. The next lemmas show that $(\II)$ and $(\III)$ are negligible at this scale.

\begin{lemma}\label{lem:II}
    $\sup_{Q \in \D_K} \mathbb{E}[\W_{p,Q}^p( \delta f,\kappa g)] \ll n^{ - p/d} \log(n)^{\alpha d}. $
\end{lemma}

\begin{lemma}\label{lem:III}
    $\sup_{Q \in \D_K} \mathbb{E}[\W_{p,Q}^p( \mu_n,\gamma \nu_n)] \ll n^{ - p/d} \log(n)^{\alpha d}$.
\end{lemma}

\subsection{Proof of the main theorem}\label{sec:proofmain}

These lemmas together with \eqref{threeterms} imply that for $n$ large enough, 
\begin{equation} 
    \forall Q \in \D_K, \quad \mathbb{E}[\W_{p,Q}^p(\mu_n , \kappa g)] \leqslant (1+\eps)^{4+p} \beta_{p,d}n^{1 - p/d} \int_{Q} g f^{-p/d}.
\end{equation}
Summing over all $Q \in \D_K$ as in the right-hand side of \eqref{ineq:mainterm} then gives
$$\limsup_{n \to \infty} \sum_{Q \in \D_K} n^{p/d - 1}\mathbb{E}[\W_{p,Q}^p(\mu_n , \kappa_Q g)] \leqslant (1+\varepsilon)^{4+p} \beta_{p,d}\int_{(0,1)^d} g f^{-p/d}.$$
This is valid for all $\varepsilon > 0$, so that we obtain \eqref{eq:main_1}. To obtain \eqref{eq:main_0}, we note that 
$$\W_{p}^p(\hat{g}_n, g) = Z_n^{-1} \W_{p}^p(\mu_n, Z_n g) = (nZ_n^{-1}) n^{-1}\W_{p}^p(\mu_n, Z_n g).$$

Lemma \ref{lem:kappa_concentration} with $Q = (0,1)^d$ shows that $\mathbb{P}(|Z_n / n - 1| > \varepsilon) \leqslant 2e^{-nC}$ for some constant $C$ depending on $\varepsilon$. Therefore, using $\W_{p,Q}(\mu_n, Z_n g)^p \les Z_n\les n$, 
we get
\begin{align*}
    \mathbb{E}[\W_{p}^p(\hat{g}_n, g)] &\leqslant (1+\varepsilon)n^{-1}\mathbb{E}[\W_{p}^p(\mu_n, Z_n g)] 
    + C\mathbb{E}[n\mathbf{1}_{|Z_n / n - 1| > \varepsilon}]\\
    &\leqslant (1+\varepsilon)n^{-1}\mathbb{E}[\W_{p}^p(\mu_n, Z_n g)] + Ce^{-nC}.
    \end{align*}
    The final conclusion \eqref{eq:main_0} follows from \eqref{eq:main_1}.

    \section{Proof of the lower bound}\label{sec:lowerbound}

\subsection{The Boundary Wasserstein distance}

For a bounded convex set $\Omega\subset \R^d$ and two positive measures $\mu, \nu$ we define the boundary Wasserstein distance as
\begin{equation}\label{defWb}
 \Wb_{p,\Omega}(\mu,\nu)=\left(\inf_{\Pi\in  \mathscr{C}_\Omega(\mu, \nu)}\int_{\overline{\Omega}\times\overline{\Omega}}|x-y|^p d\Pi(x,y)\right)^\frac{1}{p}
\end{equation}
where $\mathscr{C}_\Omega(\mu, \nu)$ is the set of couplings $\Pi$ such that $\Pi_1\restr \Omega=\mu\restr \Omega$ and $\Pi_2\restr \Omega=\nu\restr\Omega$, see \cite{figalli2010new}.
 Let us point out that as opposed to $\W_p$, we do not need to require $\mu(\Omega)=\nu(\Omega)$ 
for $\Wb_{p,\Omega}(\mu,\nu)$ to be well-defined. However, if both measures have the same mass, we always have

\begin{equation}\label{ineq:wb-w}
    \Wb_{p,\Omega}^p(\mu,\nu)\leqslant \W_p(\mu,\nu)^p.
\end{equation}
As shown in, e.g., \cite{caglioti2024subadditivity,huesmann2024asymptotics,goldman2024optimal}, the boundary Wasserstein distance is super-additive: for any collection of disjoint bounded convex sets $A_1, \ldots, A_k \subset \Omega$,
\begin{equation}\label{eq:superad}
 \Wb_{p,\Omega}^p(\mu,\nu)\geqslant \sum_{i=1}^k \Wb_{p,A_i}^p(\mu,\nu).
\end{equation}

\subsection{The lower constant}

We defined $\beta_{p,d}$ in \eqref{def:beta}. We may define the constant $\bD$ in a similar way, using the boundary Wasserstein distance. 
For an integer $n$ and a family of iid random variables $X_i$ uniformly distributed in $Q_0=(0,1)^d$ we set
\begin{equation}\label{defFb}
 Fb_p(n)=n^{p/d}\EE\sqa{\Wb_{p,Q_0}^p\left(\frac{1}{n}\sum_{i=1}^n \delta_{X_i},1\right)}.
\end{equation}
Then, from \cite{caglioti2024subadditivity}, the following limit exists and belongs to $]0, \infty[$:
\begin{equation}\label{def:betalow}
 \bD=\lim_{n\to \infty} Fb_p(n).
\end{equation}

\subsection{The lower bound}

 We aim to prove that under Hypothesis \ref{hyp:W} on $f$ and $g$, we have
\begin{equation}\label{mainclaimlower}
 \liminf_{n\to \infty} n^{p/d-1} \EE\sqa{\W_{p}^p(\mu_n,Z_ng)}\geqslant \bD \int_{(0,1)^d} g f^{-p/d}.
\end{equation}
While we could in principle argue along the same lines as for the upper bound \eqref{eq:main_1}, as shown in \cite[Theorem A.5]{caglioti2024subadditivity} the lower bound \eqref{mainclaimlower} is actually simpler. Fix $k\in \N$ with  $k\gg1$  and set $L=2^k$. For $n\in \N$ we set 
\[
K = \lfloor d^{-1} \log_2 n \rfloor
\]
which is not exactly the same as in \eqref{def:K}, 
and we define
\[
  \ell= L 2^{-K}=2^{k-K}\sim L n^{-1/d}.
\]
We then consider the partition $\D$ of dyadic cubes with side length $\ell$. Using \eqref{ineq:wb-w} and the super-additivity of $\Wb_{p,Q_0}$, we have
\begin{equation}\label{superad}
 \EE\sqa{\W_{p}^p(\mu_n,Z_ng)}\geqslant \sum_{Q\in \D} \EE\sqa{\Wb_{p,Q}^p(\mu_n,Z_ng)}.
\end{equation}
We now bound from below $\EE\sqa{\Wb_{p,Q}(\mu_n,Z_ng)^p}$ for every fixed $Q$. We recall that $\nu_n=\sum_{i=1}^n \delta_{X_i}$ and that $\kappa=\kappa_Q, \gamma=\gamma_Q, \delta=\delta_Q$ were defined in \eqref{def:kappa_gamma_delta}. We also introduce
\[
  \lambda=Z_n\frac{\delta}{\kappa}=Z_n \frac{\int_Q g}{\int_Q f}.
\]
By Young's inequality and $\Wb_{p,Q}\le \W_{p,Q}$ for measures with the same mass, we have for $\eps\in(0,1)$,
\begin{equation}\label{splitDir}
 \Wb_{p,Q}^p(\mu_n,Z_ng)\geqslant (1-\eps)\Wb_{p,Q}^p(\gamma\nu_n,\lambda f) -\frac{C}{\eps^{p-1}}\left(\W_{p,Q}^p(\mu_n,\gamma \nu_n) + \W_{p,Q}^p(\lambda f,Z_ng)\right).
\end{equation}
Notice that the main difference with \eqref{threeterms} is that here $\gamma \nu_n$ and $\lambda f$ do not have the same mass.\\
We first state the estimates for  the last two terms, see Section \ref{proofs_technical_lemmas} for the proofs. 

\begin{lemma}\label{lem:wb0}
$ \EE\sqa{\W_{p,Q}^p(\lambda f,Z_ng)}\les n \ell^{2p} |Q|$ and 
$\EE\sqa{\W_{p,Q}(\mu_n,\gamma \nu_n)^p}\les \ell^{p+1} n |Q|$.
\end{lemma}

\begin{proof}
    For the first part, we argue as in the proof of Lemma \ref{lem:II}: 
\begin{equation}\label{IIb}
 \EE\sqa{\W_{p,Q}^p(\lambda f,Z_ng)}=\EE\sqa{Z_n} \W_{p,Q}^p(\delta/\kappa f,g)\les n \ell^{2p} |Q|.
\end{equation}
For the second part, we argue exactly as in the proof of Lemma \ref{lem:III}. 
\end{proof}

We now state the estimate for the first and main term, see again Section \ref{proofs_technical_lemmas} for the proof.
\begin{lemma}\label{lem:wb1}
    
For $n$ large enough, we have
\begin{equation}\label{claimmaintermlower}
 n^{p/d-1}\EE\sqa{\Wb_{p,Q}^p(\gamma\nu_n,\lambda f)}\geqslant (1-\eps)(1-o_{L}(1))  \bD \int_{Q} g f^{-p/d} -C_\eps o_L(1) |Q|,
\end{equation}
where $\lim_{L\to \infty } o_L(1)=0$. 
\end{lemma}

Combining the results of both lemmas with \eqref{splitDir} and summing over $Q$ yields
\[
 n^{p/d-1}\EE\sqa{\W_{p}^p(\mu_n,Z_ng)}\geqslant (1-\eps)(1-o_{L}(1))  \bD \int_{Q_0} g f^{-p/d} -C_\eps \lt( o_L(1) + C_L(n^{-p/d}+n^{-1/d} )\rt).
\]
Sending first $n\to \infty$ we get
\[
 \liminf_{n\to \infty} n^{p/d-1}\EE\sqa{\W_{p}^p(\mu_n,Z_ng)}\geqslant (1-\eps)(1-o_{L}(1))  \bD \int_{Q_0} g f^{-p/d} -C_\eps o_L(1).
\]
Sending then $L\to \infty$ and finally $\eps\to 0$ concludes the proof of \eqref{mainclaimlower}. Once \eqref{mainclaimlower} is established, we can deduce that 
\begin{equation}\label{eq:main_0}
 \limsup_{n \to \infty} n^{p/d} \mathbb{E}\W_p^p\left({Z_n}^{-1}\mu_n, g\right) \leqslant \beta_{p,d}\int_{(0,1)^d} g(x) f(x)^{-p/d} \mathrm{d} x
\end{equation}
using the same arguments as in the proof of \eqref{eq:main_0} in Section \ref{sec:proofmain}.

    \section{Proofs of the technical lemmas}\label{proofs_technical_lemmas}

    \subsection{Proof of Lemma \ref{lem:reduction_to_maximal_scale}}

We start by writing a multi-scale decomposition of $\W_p^p(\mu_n, Z_n g)$. 
    
    \begin{lemma}\label{lem:1}For any $p\geqslant1$ and $0< \varepsilon\leqslant 1$, there is a constant $c_p$ such that
    \begin{equation}\label{msdecomp}\W_p^p(\mu_n, Z_n g)\leqslant (1 + \varepsilon)\sum_{Q \in \D_K}
         \W_{p, Q}^p(\mu_n, \kappa_Q g)+ \underbrace{\frac{c_p K^p}{\varepsilon^{p-1}}\sum_{k=0}^{K-1}
          \W_p^p(\lambda^k,\lambda^{k+1})}_{=:E}.
    \end{equation}
    \end{lemma}

    \begin{proof}We use the triangle inequality and \eqref{ineq:elementary}:
    \begin{align*}
    \W_{p}(\mu_n, Z_n g)&\leqslant \lt(\W_{p}(\mu_n , \lambda^K)+ \sum_{k=0}^{K-1} \W_{p}(\lambda^k,\lambda^{k+1})\rt)^p\\
     &\leqslant (1+\eps) \W_{p}^p(\mu_n , \lambda^K) + \frac{c_p}{\eps^{p-1}} \lt( \sum_{k=0}^{K-1} \W_{p}(\lambda^k,\lambda^{k+1})\rt)^p.
    \end{align*}
    By subadditivity of the Wasserstein cost, $\W_{p}^p(\mu_n , \lambda^K)\le  \sum_{Q \in \D_K} \W_{p,Q}^p(\mu_n , \kappa_{Q} g)$, which settles the first term. The second term comes from \eqref{ineq:elementary_sum}.
    
    \end{proof}
    
    To prove Lemma \ref{lem:reduction_to_maximal_scale}, we only need to prove that the error term $E$ in \eqref{msdecomp} is negligible, that is 
        \begin{equation}\label{eq:E}
            \limsup_{n \to \infty} n^{p/d - 1}E = 0.
        \end{equation}

The proof relies on the following estimate.

\begin{lemma}\label{lem:pre-E}For all $k \in \{0, \ldots, K-1\}$,
    \begin{eqnarray}
        \W_{p}^p(\lambda^k,\lambda^{k+1}) \les \ell_k^{p(1-d/2)} n^{1-\frac{p}{2}} \int_{Q_0} g^p.
    \end{eqnarray}
\end{lemma}
    
\begin{proof}
        For every fixed $k$, we have by subadditivity
        \begin{equation}\label{eq:subadditivity_negligible_term}
         \W_{p}^p(\lambda^k,\lambda^{k+1}) \leqslant \sum_{Q \in \D_k} \W_{p, Q}^p\left(\kappa_Q g, \sum_{\substack{ R \in \D_{k+1} \\ R  \subset Q}} \kappa_{R} g \right).
        \end{equation}
        We will simply write $\sum_{R\subset Q}$ to denote the sum over all $R \in \D_{k+1}$ such that $R \subset Q$.
        For a fixed $Q$, by taking the expectation in Proposition \ref{lem:bb} and using that $g$ is bounded from below, 
        \begin{align}
         \EE\lt[\W_{Q}^p\lt(\kappa_{Q} g,\sum_{R\subset Q} \kappa_R \mathbf{1}_R g\rt)\rt] &\les \EE\lt[ \frac{\ell_k^p}{(\inf_{Q} \kappa_{Q} g)^{p-1}}\int_{Q} \left|\sum_{R\subset Q} \kappa_R \mathbf{1}_R g - \kappa_{Q} g\right|^p\rt] \nonumber \\
         &\les \ell_k^p \int_{Q} \EE\lt[ \frac{1}{\kappa_{Q}^{p-1}}\left|\sum_{R\subset Q} \mathbf{1}_R\kappa_R  - \kappa_{Q} \right|^p\rt]|g|^p.\label{decomp_kappa_term}
        \end{align}
        Now, since $\kappa_Q = \sum_{R \subset Q}\mathbf{1}_R \kappa_Q$, we decompose: 
        \begin{align}
            \EE\lt[ \frac{1}{\kappa_{Q}^{p-1}}\left|\sum_{R\subset Q} \kappa_R \mathbf{1}_R  - \kappa_{Q} \right|^p\rt]&= 
            \EE\lt[\kappa_{Q} \left|\sum_{R\subset Q}  \mathbf{1}_R \left( \frac{\kappa_R}{\kappa_{Q}} - 1 \right)\right|^p\rt] \nonumber \\
            &\les \sum_{R \subset Q} \mathbf{1}_R\EE\lt[\kappa_{Q}| \kappa_R /\kappa_{Q} -1|^p\rt].\label{eq:decomp_kappa_term_2}
        \end{align}
        Lemma \ref{lem:kappa_concentration} proves that\footnote{$v_R$ was defined in \eqref{def:v}.}
$\mathbb{E}[\kappa_Q |\kappa_R/\kappa_Q - 1|^p] \leqslant n^{1 - p/2} v_R^{p/2}$. 
        Plugging this into \eqref{decomp_kappa_term}, we get
        \[
         \EE\lt[\W_{p,Q}^p\left(\kappa_{Q} g,\sum_{R\subset Q} \kappa_R g\right)\rt]\les  \ell_k^p n^{1-\frac{p}{2}}  \sum_{R\subset Q} \int_R v_R^{p/2} g^p.
        \]
        Summing over $Q \in \D_k$ we find
        \[
         \EE\lt[ \W_{p,Q^0}^p(\lambda^k,\lambda^{k+1})\rt]\les \ell_k^p n^{1-\frac{p}{2}} \sum_{R \in \D_{k+1}}\int_{R} v_R^{p/2} g^p.
        \]
        Since $v_R \les |R|^{-1}\les \ell_k^{-d}$, we can bound the right-hand side by 
        $$\ell_k^{p(1-d/2)} n^{1-\frac{p}{2}} \int_{Q_0} g^p.$$
    \end{proof}

    \begin{proof}[Proof of \eqref{eq:E}]
    
        We sum the terms $\W_{p, Q^0}^p(\lambda^k,\lambda^{k+1})$ over all scales $k \in \{0, \ldots, K-1\}$ and we absorb $\int_{Q_0} g^p$ into the constant. We get
        \begin{align*}
         \sum_{k=0}^{K-1} \EE\lt[ \W_{p, Q^0}^p(\lambda^k,\lambda^{k+1})\rt]&\les n^{1-\frac{p}{2}} \sum_{k=0}^{K-1}  {\ell_k}^{p(1-\frac{d}{2})}\les n^{1-\frac{p}{2}}\ell_K^{p(1-d/2)}
        \end{align*}
        where the last estimate follows from $\ell_k = 2^{-k}$. Plugging this into the definition of $E$ in \eqref{msdecomp} then averaging, we find
        \begin{equation}\label{mainestim}
         n^{\frac{p}{d}-1}E\les K^p n^{p(\frac{1}{d}- \frac{1}{2})}\ell_K^{p(1-d/2)}.
        \end{equation}
        Using $K \les \log n$, $d\geqslant 3$, and \eqref{eq:ell_K}, we see that 
        $$n^{p/d-1}E \les (\log n)^{\alpha p(1 - d/2)+p} \to 0$$
        provided that $\alpha$ is large enough: $\alpha > 2/(d-2)$ is sufficient and this is always satisfied by $\alpha=3$ since $d\geqslant 3$.

    \end{proof}

        \subsection{Proof of Lemma \ref{lem:I}} By \eqref{wasserstein_linearity}, $(\I)$ is equal to
         $\mathbb{E}[\gamma \W_{p,Q} ( \nu_n, (\delta/\gamma) f)^p]$. We now define
        \begin{align*}&\mathscr{A}_Q = \left\lbrace \gamma_Q > (1 + \varepsilon) \frac{\int_Q g}{\int_Q f} \right\rbrace &&\mathscr{A} = \bigcup_{Q \in \D_K} \mathscr{A}_Q \\
        &\mathscr{B}_Q = \left\lbrace \left|\frac{\nu_n(Q)}{n\int_Q f} - 1\right| > \varepsilon \right\rbrace &&\mathscr{B} = \bigcup_{Q \in \D_K} \mathscr{B}_Q
        \end{align*} 
        and finally $\mathscr{E} = \mathscr{A} \cup \mathscr{B}$. Then, 
        \begin{align}
        \mathbb{E}[\gamma \W_{p,Q}^p ( \nu_n, (\delta/\gamma) f)] &= \mathbb{E}[\mathbf{1}_{\mathscr{E}^c} \gamma \W_{p,Q}^p ( \nu_n, (\delta/\gamma) f)] + \mathbb{E}[\mathbf{1}_{\mathscr{E}} \gamma \W_{p,Q}^p ( \nu_n, (\delta/\gamma) f)] \nonumber \\
        &\leqslant (1+\varepsilon)\frac{\int_Q g}{\int_Q f} \mathbb{E}[\mathbf{1}_{\mathscr{E}^c}\W_{p,Q}^p ( \nu_n, (\delta/\gamma) f)] + \mathbb{E}[\mathbf{1}_{\mathscr{E}} \gamma \W_{p,Q}^p ( \nu_n, (\delta/\gamma) f)].\label{eq:I_bound}
        \end{align}

        \subsubsection{\texorpdfstring{Event $\mathscr{E}$}{Event E} has small probability}

        By the union bound and the fact that $|\D_K|=2^{dK}\leqslant n$, we have $\mathbb{P}(\mathscr{E}) \leqslant n \sup_{Q \in \D_K}(\mathbb{P}(\mathscr{A}_Q)+\mathbb{P}(\mathscr{B}_Q))$. Lemma \ref{lem:gamma_concentration} ensures that $\mathbb{P}(\mathscr{A}_Q) \leqslant e^{-nC|Q|}$ for some constant $C$ depending on $\varepsilon$ only. Since $n|Q|\sim \log(n)^{\alpha d}$ (see \eqref{eq:vol_DK}) we see that $\mathbb{P}(\mathscr{A}_Q) \leqslant e^{-C \log(n)^{\alpha d}}$ which is itself smaller than $n^{-p/d-10}$ for $n$ large enough. Proposition \ref{prop:concentration_number_of_points} also ensures that $\mathbb{P}(\mathscr{B}_Q) \leqslant n^{-p/d-10}$ for $n$ large enough. By the union bound, we finally obtain 
        \begin{equation}\label{eq:prob_bad_event}
        \mathbb{P}(\mathscr{E}) \les n^{-p/d-9}.
        \end{equation}

        \subsubsection{Bounding the second term of \eqref{eq:I_bound}}
        By \eqref{wasserstein_bound_bounded_domain} and the fact that $\mathrm{diam}(Q)\leqslant 1$ we see that $\W_{p,Q} ( \nu_n, (\delta/\gamma) f)^p \leqslant \nu_n(Q) \leqslant n$. Consequently, by the Cauchy-Schwarz inequality,
        \begin{align*}
        \mathbb{E}[\mathbf{1}_{\mathscr{E}} \gamma \W_{p,Q} ( \nu_n, (\delta/\gamma) f)^p] &\leqslant  n \mathbb{E}[\mathbf{1}_{\mathscr{E}} \gamma] \leqslant n \sqrt{\mathbb{P}(\mathscr{E})}\sqrt{\mathbb{E}[\gamma^2]}.
        \end{align*}
        
        Lemma \ref{lem:gamma_squared} and \eqref{eq:prob_bad_event} imply that when $n$ is large enough, this is smaller than $n^{-p/d-7}$.
        
        \subsubsection{A uniform bound for unweighted Wasserstein cost}
        We define 
        $$F_p(k, \varrho) = \mathbb{E}[\W_{Q_0}^p(k\varrho, \eta_k)], $$
        where $\varrho$ is any probability density on $Q_0$, $k$ is an integer and $\eta_k$ is a point process with $k$ independent points whose common probability density is $\varrho$. We will then use the main result of \cite{goldman2021convergence}.

        \begin{proposition}There is a constant $\beta_{p,d} >0$ depending on $p$ and $d$ such that
        \begin{equation}
            \limsup_{k\to\infty}  k^{p/d - 1} F_p(k, 1) = \beta_{p,d}.
        \end{equation}
        \end{proposition}
        
        In particular, there is an integer $N$ such that for all $k \geqslant N$, 
        \begin{equation}\label{eq:F_p_bound_uniform}
        k^{p/d - 1} F_p(k, 1) \leqslant (1+\varepsilon) \beta_{p,d}.
        \end{equation}
        
        \subsubsection{Bounding the first term of \eqref{eq:I_bound}}
        
        The proof is analogous to \cite{ambrosio2022quadratic,BeCa}. We note that $\delta / \gamma = \nu_n(Q) / \int_Q f$. By scaling, we have
        \begin{equation}
        \mathbb{E}[\W_{p,Q}^p( \delta/\gamma f, \nu_n) \mid \nu_n(Q) = k] = \ell^{p}F_p(k, f_Q)
        \end{equation}
        where $f_Q$ is the restriction of $f$ on $Q$, rescaled to $Q_0$: 
        \begin{equation}\label{eq:f_Q}
        \forall x \in Q_0, \quad f_Q(x) = \frac{\ell^d f(z + x\ell)}{\int_Q f}
        \end{equation}
        where $z$ is the corner of $Q$ with the smallest coordinates (in 2d, that would be the bottom left corner). The gradient of $f_Q$ is $\les \ell$, and there is at least one $x_0 \in (0,1)^d$ such that $f_Q(x_0)=1$, hence for any $x \in Q_0$, we have $|f_Q(x) -1| \les \ell |x - x_0| \les \ell$. The density $f_Q$ thus satisfies $|f_Q - 1|_\alpha \leqslant c \ell$ when $n$ is large enough (see \eqref{def:alpha_holder_norm} for the definition of the $\alpha$-Hölder norm). Lemma \ref{lem:transport_map} ensures that there is a map $T : (0,1)^d \to (0,1)^d$ preserving the boundary of $(0,1)^d$ and such that $T_\# f_Q = 1$ and $\mathrm{Lip}(T), \mathrm{Lip}(T^{-1}) \leqslant 1+ c_{\alpha,d}|f_Q - 1|_\alpha =: L$. 
        
        For any $(X,Y)$ a coupling of $k f_Q$ and $\eta_k$, we thus see that $(T(X), T(Y))$ is a coupling of $k\times 1$ and a point process with $k$ independent points whose common probability density is $1$. Consequently, 
        $$\mathbb{E}[|X-Y|^p] = \mathbb{E}[|T^{-1}T(X)-T^{-1}T(Y)|^p] \leqslant L^p \mathbb{E}[|T(X)-T(Y)|^p].$$
        We conclude that 
        \begin{equation}\label{eq:W_p_bound_uniform}
            \mathbb{E}[\W_{p,Q}^p(k f_Q, \eta_k) \mid \nu_n(Q) = k] \leqslant L^p F_p(k, 1).
        \end{equation}
        
        On the event $\mathscr{B}^c$, we know that for all $Q \in \D_K$, $\nu_n(Q)\geqslant (1-\varepsilon)n\int_Q f$, which by Hypothesis \ref{hyp:W} is $\ges \log(n)^{\alpha d}$. In particular, when $n$ is large enough, on the event $\mathscr{B}^c$, we have $\nu_n(Q)\geqslant N$ with
         $N$ defined for \eqref{eq:F_p_bound_uniform}. Consequently, uniformly over all $Q \in \D_K$, we have
        \begin{align*}
        \mathbb{E}[\mathbf{1}_{\mathscr{E}^c}\W_{p,Q}^p( \delta/\gamma f, \nu_n)] &\leqslant  \mathbb{E}[\mathbf{1}_{\mathscr{B}^c} \mathbb{E}[\W_{p,Q}^p( \delta/\gamma f, \nu_n) \mid \nu_n(Q)]] \\
        & \leqslant \ell^p L^p \mathbb{E}[\mathbf{1}_{\mathscr{B}^c} F_p(\nu_n(Q), 1)] \\
        & \leqslant \ell^p L^p \mathbb{E}\left[\mathbf{1}_{\mathscr{B}^c} (1+\varepsilon)\nu_n(Q)^{1 - p/d}\beta_{p,d}\right] \\
        &\leqslant (1+\varepsilon)^{2 - p/d}\ell^p L^p \left(n \int_Q f\right)^{1 - p/d}\beta_{p,d}\\
        & = (1+\varepsilon)^{p+1}\ell^p  n^{1 - p/d}\beta_{p,d}\left(\int_Q f\right)^{1 - p/d}
        \end{align*}
        where in the last line, we used that $L-1 \les \ell$, hence $L \leqslant 1+\varepsilon$ when $n$ is large enough.
        
        \subsubsection{Gathering everything}
        Let us go back to \eqref{eq:I_bound}. If $n$ is large enough, then for all $Q \in \D_K$, we have
        \begin{align*}
        \mathbb{E}[\W_{p,Q}^p ( \gamma \nu_n, \delta f)] &\leqslant (1+\varepsilon)\frac{\int_Q g}{\int_Q f} \mathbb{E}[\mathbf{1}_{\mathscr{E}^c}\W_{p,Q}^p ( \nu_n, (\delta/\gamma) f)] + n^{-p/d-7} \\
        &\leqslant (1+\varepsilon)^{1+p} n^{1 - p/d} \beta_{p,d} \frac{\int_Q g}{\left(\int_Q f \right)^{p/d}} + n^{-p/d-7}
        \end{align*}
        By the uniform continuity of $f$ over $Q_0$, for all $x \in Q$ we have 
        \begin{equation*}
        \frac{\int_Q g}{\left(\int_Q f\right)^{p/d}} \leqslant (1+\varepsilon)g(x)f(x)^{-p/d}.
        \end{equation*}
        Plugging this into the preceding equation, we obtain that (uniformly over all $Q \in \D_K$), when $n$ is large enough, we have
        \begin{align*}
        \mathbb{E}[\W_{p,Q}^p ( \gamma \nu_n, \delta f)] &\leqslant (1+\varepsilon)^{2+p} n^{1 - p/d} \beta_{p,d} \int_Q gf^{-p/d} + n^{-p/d-7}\\
        &\leqslant (1+\varepsilon)^{3+p} n^{1 - p/d} \beta_{p,d} \int_Q gf^{-p/d}
        \end{align*}
        where in the last line we used the fact that $\int_Q gf^{-p/d} \ges \int_Q g \ges n^{-1}$, and thus $n^{-p/d-7}$ is negligible with respect to $n^{1 - p/d} \int_Q gf^{-p/d}$.

    \subsection{Proof of Lemma \ref{lem:II} and the first bound of Lemma \ref{lem:wb0} }
    We prove the following general bound: for any $Q \subset (0,1)^d$,
    \begin{equation}\label{eq:general_bound_II}
        \EE[\W_{p,Q}^p(\delta f, \kappa g)]   \les  n \ell^{2p} |Q|.
    \end{equation}
    Applying this first to $\ell=\ell_K$  and observing that  $\ell_K = (\log n)^{\alpha} n^{-1/d}$, this would prove Lemma \ref{lem:II}. Moreover, this proves also the first bound in Lemma \ref{lem:wb0}.

    \begin{proof}
    
    By \eqref{wasserstein_linearity}, we have $\W_{p,Q}^p(\delta f, \kappa g)  = \kappa \W_{p,Q}^p((\delta/\kappa) f, g)$. By noting that $\mathbb{E}[\kappa] = n$ and that $\delta/\kappa = \int_Q g / \int_Q f$, we get
    \begin{align*}
        \mathbb{E}[\W_{p,Q}^p(\delta f, \kappa g)] \leqslant \mathbb{E}[\kappa \W_{p,Q}^p((\delta/\kappa) f, g)] = n \W_{p,Q}^p((\delta/\kappa) f, g).
    \end{align*}
    We now use the bound from Proposition \ref{lem:bb}:  
    \begin{align*}
       (\III) = \mathbb{E}[\W_{p,Q}^p(\delta f, \kappa g)] &\les  \frac{n\ell_K^p}{(\inf g)^{p-1}} \int_Q \left|g - (\delta/\kappa)f \right|^p \les n\ell_K^{2p} \int_Q |\nabla (g - (\delta/\kappa)f)|^p, 
    \end{align*}
    where we used the fact that $(\inf g)^{1-p} \les 1$ by Hypothesis \ref{hyp:W}.
    The Poincaré-Wirtinger inequality (Proposition \ref{prop:poincare}) applied to $g - (\delta/\kappa)f$, which has mean value 0 on $Q$, yields 
    $$(\III) \les n\ell_K^{2p} \int_Q |\nabla (g - (\delta/\kappa)f)|^p.$$
    Since $\nabla g$ and $\nabla f$ are bounded and $\delta/\kappa$ is also bounded due to Hypothesis \ref{hyp:W}, this whole term has order $n\ell_K^{2p} |Q|$. 
    \end{proof}

    \subsection{Proof of Lemma \ref{lem:III} and second bound of Lemma \ref{lem:wb0} }
    We prove the more general bound: for any $Q \subset (0,1)^d$,
    \begin{equation}\label{eq:general_bound_III}
        \EE[\W_{p,Q}^p(\mu_n, \gamma \nu_n)] \les n \ell^{p+1} |Q|.
    \end{equation}
    Applying this first to $\ell=\ell_K$  this would conclude the proof of Lemma \ref{lem:III}. Moreover, this proves also the second bound in Lemma \ref{lem:wb0}.\\

    We start with the simple case where $\gamma=1/n$ so that both $\mu_n$ and $\gamma \nu_n$ are probability measures. For any coupling $(X,Y)$ of $\mu_n$ and $\gamma\nu_n$, the cost of this coupling is equal to 
            $$\sum_{x, y}\mathbb{P}(X=x, Y=y) |x-y|^p \leqslant \mathrm{diam}(Q)^p \mathbb{P}(X\neq Y)$$
            which proves that $ \W_{p,Q}^p(\mu_n , \gamma \nu_n) \leqslant \mathrm{diam}(Q)^p \mathrm{d}_{\mathrm{TV}}(\mu_n, \gamma \nu_n)$. This remains true for any $\gamma$ by scaling. Now, the measures $\mu_n$ and $\gamma \nu_n$ are supported on the same discrete set, so their total variation distance is equal to
            $$\dtv(\mu_n, \gamma \nu_n) = \frac{1}{2}\sum_{X_i \in Q}|W_i- \gamma|.$$
    
    We also have $W_i - \gamma = \nu_n(Q)^{-1}\sum_{X_j \in Q} (W_i - W_j)$, so that 
    \begin{align*}\dtv(\mu_n, \gamma \nu_n) &\leqslant \frac{\nu_n(Q)^{-1}}{2}\sum_{X_i, X_j \in Q} |W_i - W_j|.
    \end{align*}
Remember that $W_i = g(X_i)/f(X_i)$; thanks to Hypothesis \ref{hyp:W}, the function $g/f$ is Lipschitz. Indeed, 
$$\left|\nabla \frac{g}{f}\right| = \left|\frac{\nabla g}{f} - \frac{g}{f}\frac{\nabla f}{f}\right| \leqslant c^2 + c^4 \leqslant 2c^4.$$
We thus have
\begin{align*}\mathbb{E}\dtv(\mu_n, \gamma \nu_n) &\les \mathbb{E}\frac{\nu_n(Q)^{-1}}{2}\sum_{X_i, X_j \in Q} |X_i - X_j| \\
    &\les \mathbb{E}\nu_n(Q)\mathrm{diam}(Q)\\
    &\les n\ell |Q|.
\end{align*}
            This concludes the proof of \eqref{eq:general_bound_III}.

    \subsection{Proof of Lemma \ref{lem:wb1}}

Lemma \ref{lem:transport_map} gives a map $T$ which is a $(1+\eps)$-bi-Lipschitz map on $Q$ and transports $f/\int_{Q} f$ to $1/|Q|$. Let $N=\nu_n(Q)$ and let $\hat{\nu}_N=T\sharp \nu_n$ so that
    \[
     \hat{\nu}_N=\sum_{i=1}^N \delta_{Y_i}
    \]
    with $Y_i=T(X_i)$ which are iid uniformly distributed in $Q$. We then have
    \[
     \Wb_{p,Q}^p(\gamma\nu_n,\lambda f)\geqslant (1-\eps)^p \Wb_{p,Q}^p\left(\gamma\hat{\nu}_N,\frac{Z_n}{ |Q|} \int_Q g\right).
    \]
    We set $M=Z_n  \gamma^{-1} \int_Q g\sim Z_n |Q|$. Using the homogeneity of $\Wb_{p,Q}$, the triangle inequality, and the Young inequality, we see that
    \begin{multline}\label{splitmainterm}
     \Wb_{p,Q}^p\left(\gamma\hat{\nu}_N, \frac{Z_n}{ |Q|} \int_Q g\right)=\gamma \Wb_{p,Q}^p\left(\hat{\nu}_N,M/|Q|\right)\\
     \geqslant \gamma\lt[(1-\eps) \Wb_{p,Q}^p(\hat{\nu}_N,N/|Q|)-\frac{C}{\eps^{p-1}}\Wb_{p,Q}^p(N/|Q|, M/|Q|)\rt].
    \end{multline}
    To estimate the last right-hand side term we appeal to \cite[Lemma 2.5]{caglioti2024subadditivity} to obtain
    \[
     \Wb_{p,Q}^p(N/|Q|, M/|Q|)\les \ell^{p+d} \frac{|N/|Q|- M/|Q||^p}{(M/|Q|)^{p-1}}=\ell^{p} \frac{|N-M|^p}{M^{p-1}}.
    \]
    We now write that
    \begin{multline*}
     |N-M|= \frac{N}{\mu_n(Q)}|\mu_n(Q)-Z_n\int_Q g|\le \frac{N}{\mu_n(Q)}|\mu_n(Q)-n\int_Q g| + \frac{N}{\mu_n(Q)} |n-Z_n|\int_Q g\\
     \les |\mu_n(Q)-n\int_Q g| +  |n-Z_n||Q|=|\mu_n(Q)-\EE\sqa{\mu_n(Q)}| +  |Z_n-\EE\sqa{Z_n}||Q|
    \end{multline*}
    to obtain by Cauchy-Schwarz,
    \begin{multline*}
     \EE\sqa{\frac{|N-M|^p}{M^{p-1}}}\\
     \les |Q|^{1-p} \EE\sqa{Z_n^{2(1-p)}}^{1/2}\lt(\EE\sqa{|\mu_n(Q)-\EE\sqa{\mu_n(Q)}|^{2p}}^{1/2}+ |Q|^p\EE\sqa{|Z_n-\EE\sqa{Z_n}|^{2p}}^{1/2}\rt)\\
     \les |Q|^{1-p} n^{1-p}\lt( (n|Q|)^{p/2}+ |Q|^pn^{p/2}\rt)
     \les (|Q| n)^{1-p/2}.
    \end{multline*}
    In conclusion, recalling that by definition $n|Q|=L^d$,
    \begin{align}
     \EE\sqa{\Wb_{p,Q}^p(N/|Q|, M/|Q|)}&\les n^{1-p/d} (|Q|n)^{p(\frac{1}{d}-\frac{1}{2})} |Q|\label{estimchangemass}\\
     &=n^{1-p/d} L^{-p(\frac{d}{2}-1)} |Q|\\
     &=n^{1-p/d}o_L(1) |Q|.
     \end{align}

    We now estimate the first right-hand side term of \eqref{splitmainterm}. By the uniform continuity of $f$ and $g$, when $n$ is large enough we have $\gamma\geqslant (1-\eps) g(x)/f(x)$. 
    Recalling the definition \eqref{defFb} of $Fb_p$, we thus have by scaling
    \[
     \gamma \EE\sqa{\Wb_{p,Q}^p(\hat{\nu}_N,N/|Q|)}\geqslant (1-\eps) \frac{g(x)}{ f(x)} \ell^p \EE\sqa{N^{1-p/d}Fb_p(N)}.
    \]
    Since $\EE\sqa{N}=n\int_Q f\geqslant (1-\eps) f(x) n|Q|$ which tends to infinity as $L\to \infty$ we have by definition of $\bD$ and the good concentration properties of $N$ that
    \[
     \EE\sqa{N^{1-p/d}Fb_p(N)}\geqslant (1- o_L(1))\bD \lt(n\int_Q f\rt)^{1-p/d}
    \]
    and thus
    \begin{align*}
     \gamma \EE\sqa{\Wb_{p,Q}^p(\hat{\nu}_N,N/|Q|)}&\geqslant (1-\eps) (1- o_L(1))\bD \frac{g(x)}{ f(x)} \ell^p (n|Q| f(x))^{1-p/d}\\
     &=n^{1-p/d}(1-\eps) (1- o_L(1))\bD g(x) f(x)^{-p/d} |Q|.
    \end{align*}
    Averaging over $x$ we find
    \begin{equation}\label{maintermlocDir}
     \gamma \EE\sqa{\Wb_{p,Q}^p(\hat{\nu}_N,N/|Q|)}\geqslant n^{1-p/d}(1-\eps) (1- o_L(1))\bD \int_Q g f^{-p/d}.
    \end{equation}
    Combining \eqref{estimchangemass} and \eqref{maintermlocDir} together yields \eqref{claimmaintermlower}.

\bibliographystyle{abbrv}

\bibliography{OT}

\appendix

\section{Concentration inequalities}\label{appendix:concentration_inequalities}

We need concentration results for sums of random variables, typically of the form 
$$\sum_{i=1}^n W_i \mathbf{1}_{X_i \in Q}$$
where $Q$ is a very small measurable set --- in particular, most of the summands are zero, and the variance of the summands is small. In this context, Bennett-type inequalities are tighter than Hoeffding's. For generality, we will formulate our results for any pair $(X,W)$ of random variables, even if in the main text we will only use the case where $X$ is a random variable with density $f$ and $W = g(X)/f(X)$.

\subsection{Bennett's inequality}
Let $Z_1, \ldots, Z_n$ be a sequence of iid random variables and let $Z = \sum_{i=1}^n Z_i$. Then, Bennett's inequality states that for any $t>0$, 
\begin{equation}
    \mathbb{P}(Z > \mathbb{E}[Z] + t) \leqslant \exp\left(-n (\sigma/m)^2 \times h\left(\frac{mt}{n\sigma^2}\right)\right), 
\end{equation}
where $m = |Z_i|_\infty$ and $\sigma$ is the standard deviation of $Z_i$. The function $h$ is defined by
\begin{equation}
    h(x) = (1+x)\ln(1+x) - x.
\end{equation}
For future reference, we state the elementary properties of $h$ as a lemma.

\begin{lemma}\label{lem:h}
    $h$ is convex and increasing, $h'(x) = \ln(1+x)$, and $h(x) \geqslant \max\{0, (x^2 - x^3)/2\}$ for all $x\geqslant 0$.
\end{lemma}

\subsection{Concentration for linear functionals}

Let $(X_i)$ be iid random variables with density $f$ and let $W_i = g(X_i)/f(X_i)$. For any bounded measurable function $\varphi$, we are interested in the concentration of the functionals 
$$I(\varphi)  = \sum_{i=1}^n W_i \varphi(X_i). $$
The expectation of $I$ is given by $\mathbb{E}[I(\varphi)] = n\mathbb{E}_{X\sim g}[\varphi(X)]= n\int g\varphi$. 
Directly applying Bennett's inequality, we get the following concentration result. 
\begin{proposition}
    Let $\varphi$ be a bounded measurable function. We suppose that $W\varphi(X)$ is bounded by a constant $m_\varphi:=|W\varphi(X)|_\infty$. Then, noting $\sigma^2_\varphi = \mathrm{Var}(W \varphi(X))$, we have 
    \begin{equation}\label{ineq:concentration_sum_functional}
        \mathbb{P}\left(\sum_{i=1}^n W_i \varphi(X_i) > (1+\delta) n\int g(x)\varphi(x)\,dx\right) \leqslant \exp\left(-n \frac{\sigma^2_\varphi}{m_\varphi^2} h\left(\frac{\delta m_\varphi }{ \sigma^2_\varphi} \int g(x)\varphi(x)\,dx\right)\right).
    \end{equation}
    where $h(x) = (1+x)\ln(1+x) - x$.
    The same inequality holds for the lower deviation. 
\end{proposition}
\begin{proof}The result is a straightforward application of Bennett's inequality, choosing $\xi = \delta n\int g \varphi$. \end{proof}

\subsection{Concentration for $\kappa$} 

We will use the precedings bounds with $\varphi = \mathbf{1}_{X \in Q}$ to get concentration bounds for $\kappa_Q$. Recall that 
$$\kappa_Q  = \frac{\sum W_i \mathbf{1}_{X_i \in Q}}{\int_Q g}$$
so that $\mathbb{E}[\kappa_Q] = n$. 
For simplicity we will note $m,\sigma^2$ instead of $m_Q,\sigma^2_Q$ and we will often use the notation  $$b =  \frac{m\mathbb{E}[W \mathbf{1}_{X \in Q}]}{\sigma^2}.$$ 
The preceding concentration results thus yield that 
\begin{equation}\label{i:kappa}
    \mathbb{P}\left(\left| \frac{\kappa_Q}{n} - 1 \right| > \delta\right) \leqslant 2e^{-n \sigma^2/m^2 h(\delta b)}
\end{equation}
and the same inequality holds for the lower deviation. 

\begin{proposition}For any $Q \subset (0,1)^d$ and any fixed $p>0$, 
    
    \begin{equation}\label{ineq:bound_on_kappa_2}
        \mathbb{E}\left[ \left| \frac{\kappa_Q}{n} - 1 \right|^p\right] \les \mathrm{Var}\left(\frac{\kappa_Q}{n}\right)^{p/2}.
    \end{equation}
    
\end{proposition}

\begin{proof}
Set $V = \sigma^2/m^2$ to lighten the notation. Since $\kappa_Q = \int_Q w\mu / \int_Q g$, by the union bound for upper and lower deviations in \eqref{i:kappa} we have $\mathbb{P}(|\kappa_Q/n - 1| > \delta)\leqslant 2e^{-n V h(\delta b)}$. Then, for any $p>0$, 
\begin{align*}
\mathbb{E}\left[ \left| \frac{\kappa_Q}{n} - 1 \right|^p\right] &= \int_0^\infty \mathbb{P}\left( \left| \frac{\kappa_Q}{n} - 1 \right|^p > t \right) dt \\
&\leqslant \int_0^\infty 2e^{-nV h(t^{1/p}b)} dt \\
&\leqslant 2pb^{-p}\int_0^\infty e^{-N h(u)}u^{p-1} du 
\end{align*}
where $N = nV$. We split the integral in two parts, $\int_0^\infty = \int_0^{1/2} + \int_{1/2}^\infty$. 
For the first term, we use the bound on $h$ in Lemma \ref{lem:h}. For $x \leqslant 1/2$, we have $h(x) \geqslant(x^2 - x^3)/2 \geqslant x^2/ 4$, hence
$$ \int_0^{1/2}e^{-Nh(u)}u^{p-1}du \leqslant \int_0^{1/2} e^{-Nu^2/4}u^{p-1}du \les N^{-p/2}.$$
For the second term, by convexity of $h$ we have $h(u)\geqslant h(1/2)+(u-1/2)h'(1/2) = u \ln(3/2) - \ln(\sqrt{3/2})$. Plugging this estimate, we get
$$ \int_{1/2}^{\infty} e^{-Nh(u)}u^{p-1}du \les \int_{1/2}^\infty e^{-\ln(3/2)Nu}u^{p-1}du \les N^{-p}.$$
Gathering the two bounds yields 
$$\mathbb{E}\left[ \left| \frac{\kappa_Q}{n} - 1 \right|^p\right] \les b^{-p}N^{-p/2} = (b^2 Vn )^{-p/2}. $$
By the definition of $b$ and $V=\sigma^2/m^2$,  we have
$$b^2 V= \frac{\mathbb{E}[W \mathbf{1}_{X \in Q}]^2}{\mathrm{Var}(W \mathbf{1}_{X \in Q})}$$
which concludes the proof.
\end{proof}

\begin{lemma}\label{lem:kappa_concentration}Let $A\subset B$ be two subsets of $(0,1)^d$ such that $\int_A g \sim \int_B g$. Then, $$\mathbb{E}[\kappa_B |\kappa_A/\kappa_B - 1|^p]\les n^{1 - p/2}\left(\frac{\mathrm{Var}(W \mathbf{1}_{X \in A})}{\mathbb{E}[W \mathbf{1}_{X \in A}]^2}\right)^{p/2},$$ 
    where it is understood that $\kappa_A / \kappa_B=0$ if $\kappa_B=0$.
\end{lemma}

\begin{proof}We first emphasize that since $A \subset B$, then
    $$\frac{\kappa_A}{\kappa_B}  = \frac{\int_A w\mu}{\int_B w\mu}\frac{\int_B g}{\int_A g}\leqslant \frac{\int_B g}{\int_A g} \les 1.$$
    Let us note $\mathscr{E} = \{|\kappa_A - n\int_A g| > n\int_A g/2\}$, an event with probability smaller than $e^{-n (\sigma^2/m^2)h(b/2)}$ thanks to \eqref{i:kappa}. Then, 
    \begin{align*}\mathbb{E}[\kappa_B |\kappa_A/\kappa_B - 1|^p]&\les \mathbb{E}[\mathbf{1}_{\mathscr{E}}\kappa_A] + \mathbb{E}\left[\mathbf{1}_{\bar{\mathscr{E}}}\frac{|\kappa_B - \kappa_A|^p}{\kappa_B^{p-1}}\right] \\ 
    &\les  b\mathbb{P}(\mathscr{E}) +  n \mathbb{E}[|\kappa_B/n-\kappa_A/n|^p]
\end{align*}
 because $\kappa_A$ is smaller than $b$. 
 The second term is smaller than $n (\mathbb{E}[|\kappa_B/n - 1|^p] + \mathbb{E}[|\kappa_A/n - 1|^p])$ up to a constant. We now use \ref{ineq:bound_on_kappa_2} for both terms; since the variance of $\kappa_A$ is smaller than the variance of $\kappa_B$, we get an overall bound which is $\les n \mathrm{Var}(\kappa_A/n)^{p/2}$.
\end{proof}

\subsection{Concentration for $\gamma$} 

We recall that 
$$\gamma_Q = \frac{\sum_{i=1}^n W_i \mathbf{1}_{X_i \in Q}}{\sum_{i=1}^n \mathbf{1}_{X_i \in Q}} .$$

\begin{lemma}\label{lem:gamma_concentration}For any $\varepsilon > 0$ and any $Q \subset (0,1)^d$, there is a constant $C(\varepsilon)$ such that
\begin{equation}
    \mathbb{P}\left(\gamma > (1 + \varepsilon)\frac{\int_Q g}{\int_Q f}\right) \leqslant 2e^{-n|Q|C(\varepsilon)}.
\end{equation}
\end{lemma}

\begin{proof}We note $\mathscr{E} = \left\{\sum \mathbf{1}_{X_i \in Q} < (1 - \nu)n\int_Q f\right\}$ for some $\nu > 0$. Then, 
    \begin{align*}\mathbb{P}\left(\gamma > (1 + \varepsilon)\frac{\int_Q g}{\int_Q f}\right) &\leqslant \mathbb{P}(\mathscr{E}) + \mathbb{P}\left(\sum_{i=1}^n W_i \mathbf{1}_{X_i \in Q} > (1 + \varepsilon) (1 - \nu)n\int_Q g \right)\\
        &\leqslant \mathbb{P}(\mathscr{E}) + \mathbb{P}\left(\sum_{i=1}^n W_i \mathbf{1}_{X_i \in Q} > (1 + \varepsilon/2)n\int_Q g \right)
    \end{align*}
    for some $\nu > 0$ sufficiently small ($\nu=\varepsilon/10$ is enough).

    For the first term, we use our main results with weights $W_i$ equal to $1$. In this case, $m=1$, $\mathbb{E}[W \mathbf{1}_{X \in Q}] = \int_Q f=:q$ and $\sigma^2 = q(1-q)$. Hence, we obtain a bound of $e^{-n q(1-q) h(\delta /(1-q))}$. 
    For the second term, we apply \eqref{ineq:concentration_sum_functional} with $\varphi(x) = \mathbf{1}_{x \in Q}$. In this case $\mathbb{E}[W\mathbf{1}_{X \in Q}] = \int_Q g = :p$ and we can use Hypothesis \ref{hyp:W} to bound $1/c \leqslant m \leqslant c$, hence we obtain a bound of $e^{-n \sigma^2/c^2 h(\varepsilon \int_Q g/2c\sigma^2)}$. 
    Again by Hypothesis \ref{hyp:W}, it is easily seen that $\sigma^2 \leqslant p(c-p)$, and that $p,q \leqslant c|Q|$. The result follows from elementary estimates on $h$. 
\end{proof}

We will also need the following very crude bound on $\gamma$. 

\begin{lemma}\label{lem:gamma_squared}
    $\mathbb{E}[\gamma^2]\leqslant n c$. 
\end{lemma}

\begin{proof}It is clear that $\gamma$ is smaller than $\max \{W_i : X_i \in Q\}$, so $\gamma^2$ is smaller than $\sum_{i=1}^n W_i^2 \mathbf{1}_{X_i \in Q}$. Taking the expectation and using Hypothesis \ref{hyp:W} we get the result.
\end{proof}

\subsection{Concentration for the number of points}

Finally, we will need a uniform control of the number of points $X_i$ falling in each cube $Q \in \D_K$. We recall that $\nu_n(A) = \sum_{i=1}^n \delta_{X_i}(A)$ is the unweighted empirical measure of the $X_i$. Of course, the $n^{-p/d-10}$ here is somehow arbitrary; any term which is sufficiently small with respect to $n^{-p/d}$ would work.

\begin{proposition}\label{prop:concentration_number_of_points}
    With $K$ as in \eqref{def:K}, then $\forall \varepsilon > 0$, if $n$ is large enough,
    $$\mathbb{P}\left(\forall Q \in \D_K, (1-\varepsilon)n\int_Q f \leqslant \nu_n(Q) \leqslant (1+\varepsilon)n\int_Q f\right) \geqslant 1 - n^{-p/d-10}.$$
    \end{proposition}
    
    \begin{proof}
    For any fixed $Q$, $\nu_n(Q)$ has distribution $\mathrm{Bin}(n, p_Q)$ where $p_Q = \int_Q f$. The number of cubes in $\D_K$ is $2^{dK}$ which is smaller than $n$. By the union bound and Chernoff's inequality for Binomial random variables, for any $0<\delta < 1$,
    $$\mathbb{P}\left(\exists Q \in \D_K, \left|\frac{\nu_n(Q)}{np_Q}-1\right|>\delta \right) \leqslant \sum_{Q \in \D_K}\mathbb{P}\left(\left|\frac{\nu_n(Q)}{np_Q}-1\right|>\delta \right) \leqslant n \times 2e^{-\delta^2 np_Q/3}.$$
    By Hypothesis \ref{hyp:W}, we have $p_Q \geqslant c \log(n)^{\alpha d}/n$ for some $c>0$ and all $Q$. Hence, the probability bound above is smaller than $2\exp\{\log (n) - \delta^2\log(n)^{\alpha d}/3\}$ which is smaller than $n^{-p/d-10}$ for $n$ large enough. 
    \end{proof}

\section{Functional inequalities and transport}\label{appendix:functional_inequalities}

\subsection{Functional inequalities}
We use the Poincaré-Wirtinger inequality in the following form. 

\begin{proposition}\label{prop:poincare}
Let $Q \subset \mathbb{R}^d$ be an open connected set with smooth boundary. For any $p \in [1, \infty[$, there is a constant $C$ depending only on $p$ such that for any differentiable function $h$ on $Q$ such that $\nabla h \in L^p(Q)$, 
\begin{equation}
    \int_Q |h-m|^p \leqslant C \mathrm{diam}(Q)^p \int_Q |\nabla h|^p
\end{equation} 
where $m = \frac{1}{|Q|} \int_Q h$ is the mean value of $h$ on $Q$.
\end{proposition}

We also made repeated use of the following Proposition, which is proven in \cite{goldman2021convergence} (Lemma 3.4). 

\begin{proposition}\label{lem:bb}
Let $h_1, h_2$ be probability densities on a dyadic cube $Q$, with $\inf_Q h_1 >0$. Then, $\forall p\geqslant 1$,
\begin{equation}
    \W_{p,Q}(h_1, h_2)^p \leqslant \frac{\mathrm{diam}(Q)^p}{(\inf_Q h_1)^{p-1}} \int_Q |h_2 - h_1|^p.
\end{equation}
\end{proposition}

\subsection{Transport map between almost uniform measures}

We end this section by stating Proposition 2.4 from \cite{ambrosio2022quadratic}, a result on the existence of a (non-optimal) transport map between almost uniform measures. The $\alpha$-Hölder norm used in the lemma is defined by 
\begin{equation}\label{def:alpha_holder_norm}
    |h|_\alpha = \sup_{x,y \in (0,1)^d, x\neq y} \frac{|h(x) - h(y)|}{|x-y|^\alpha} + \sup_{x \in (0,1)^d} |h(x)|
\end{equation}
and the Lipschitz constant of a function $h$ is defined by 
\begin{equation}
    \mathrm{Lip}(h) = \sup_{x,y \in (0,1)^d, x\neq y} \frac{|h(x) - h(y)|}{|x-y|}.
\end{equation}

\begin{lemma}\label{lem:transport_map}
    For any $d\geqslant 1$ and $\alpha \in ]0,1]$, there exists a constant $c_{\alpha,d} > 0$ such that for any density $\rho$ on $(0,1)^d$ such that $|\rho - 1|_\alpha \leqslant 1/2$, there is a map $T : (0,1)^d \to (0,1)^d$ preserving the boundary of $(0,1)^d$ and such that $T_\# \rho = 1$ and 
    \begin{equation}
        \mathrm{Lip}(T), \mathrm{Lip}(T^{-1}) \leqslant 1+ c_{\alpha,d}|\rho - 1|_\alpha.
    \end{equation}
\end{lemma}

\noindent\hrulefill

\end{document}